\documentclass[a4paper,11pt]{article}

\usepackage[T1]{fontenc}
\usepackage[utf8]{inputenc}
\DeclareUnicodeCharacter{00A0}{~}
\usepackage{a4wide}
\usepackage{amssymb,amsmath,amscd,amsfonts,amsthm,bbm,mathrsfs,enumerate}
\usepackage[colorlinks=true, linkcolor=blue, citecolor=blue]{hyperref}
\usepackage{shortcuts}
\usepackage{color}
\usepackage{graphicx}
\usepackage{tikz,pgfplots}
\usepackage{color}

\theoremstyle{plain} 
\newtheorem{theorem}{Theorem}
\newtheorem{lemma}{Lemma}
\newtheorem{assumption}{Assumption}
\newtheorem{remark}{Remark}

\newtheorem{definition}{Definition}

\newtheorem{proposition}[theorem]{Proposition}
\DeclareMathOperator{\acc}{acc}

\title{Stochastic Proximal Subgradient Descent Oscillates in the Vicinity of its Accumulation Set}

\author{Sholom Schechtman}
\date{March 2021}
\begin{document}

\maketitle

\begin{abstract}
  We analyze the stochastic proximal subgradient descent in the case where the objective functions are path differentiable and verify a Sard-type condition. While the accumulation set may not be reduced to unique point, we show that the time spent by the iterates to move from one accumulation point to another goes to infinity. An oscillation-type behavior of the drift is established. These results show a strong stability property of the proximal subgradient descent. Using the theory of closed measures, Bolte, Pauwels and R\'ios-Zertuche \cite{bol_zer_pau20} established this type of behavior for the deterministic subgradient descent. Our technique of proof relies on the classical works on stochastic approximation of differential inclusions, which allows us to extend results in the deterministic case to a stochastic and proximal setting, as well as to treat these different cases in a unified manner.
\end{abstract}
\section{Introduction}
Let $d$ be a positive integer, let $\cX$ be a nonempty, closed and convex set and let $f,g: \bbR^d \rightarrow \bbR$ be two locally Lipschitz functions. In this note, we study the behavior of the stochastic proximal subgradient descent (SPGD):
\begin{equation}\label{eq:spgd_incl}
  x_{n+1} \in \prox^{\gamma_n}_{g, \cX}( x_n - \gamma_n v_n  + \gamma_n \eta_{n+1}) \, ,
\end{equation}
where $\prox^{\gamma_n}_{g, \cX}$ is the proximal operator for the function $g$ on $\cX$ (see Eq.~\eqref{eq:prox} for a definition), $(\gamma_n)$ is a sequence of stepsizes, $(\eta_n)$ is a noise sequence and for each $n \in \bbN$, $v_n$ is in the set $\partial f(x_n)$ of Clarke's subgradients of $f$ at $x_n$.

Let $\cN_{\cX}(x)$ be the normal cone of $\cX$ at $x$. It is known (see \cite{dav-dru-kak-lee-19}, \cite{maj-mia-mou18}) that, under mild conditions on $f$, $g$ and $(\eta_n)$, every limit point of $(x_n)$ is included in the set $\cZ := \{ x: 0 \in \partial f(x) + \partial g(x) + \cN_{\cX}(x)\}$.
 The proof leans on the seminal paper of Benaïm, Hofbauer and Sorin \cite{ben-hof-sor-05} (see also \cite{ben-(cours)99}, \cite{duchi2018stochastic}), which analyzes Eq.~\eqref{eq:spgd_incl} as an Euler-like discretization of the differential inclusion (DI):
\begin{equation}\label{eq:DI_prox}
  \dot{\sx}(t) \in - \partial f(\sx(t)) - \partial g(\sx(t)) - \cN_{\cX}(\sx(t))\, .
\end{equation}



While the sequence $(x_n)$ is known to converge to $\cZ$, recent work \cite{zer20} shows that in principle, it might not converge to a unique point. In \cite[Section 2]{zer20} R\'ios-Zertuche considers the deterministic subgradient descent (that is to say $g=0$, $\cX = \bbR^d$, $\eta_n= 0$) and constructs $f$, which verifies main assumptions of nonsmooth optimization (such as Whitney stratifiability or Kurdyka-\L{}ojasiewicz inequality) but the limit set of $(x_n)$ is equal to $\cZ= \{ x: \norm{x} =1\}$. This encourages a more precise study of Eq.~\eqref{eq:spgd_incl}.

 In \cite{bol_zer_pau20} the authors,  using the theory of closed measures, show that in the case of the deterministic subgradient descent
the convergence to $\cZ$ arises in a structured manner. First, they prove that if $x, y$ are two distinct accumulation points of $(x_n)$, then the time that the iterates spend to get from a neighborhood of $x$ to a neighborhood of $y$ goes to infinity. Second, in a first approximation their results imply that if $x$ is an accumulation point of $(x_n)$, then
\begin{equation*}\label{eq:osc}
  \frac{\sum_{i=1}^n \gamma_{i} v_i \1_{x_i \in B(x, \delta)}}{\sum_{i=1}^n \gamma_{i} \1_{x_i \in B(x, \delta)}} \xrightarrow[n \rightarrow + \infty]{} 0 \, ,
\end{equation*}
(see \cite[Th. 7]{bol_zer_pau20} or Section~\ref{sec:main} for a precise statement).
Intuitively speaking, this means that even if $x_n - x_0 = \sum_{i=0}^n \gamma_i v_i$ does not converge, on average, the drift coming from the subgradients compensate itself and vanishes at infinity. This behavior captures an oscillation phenomenon of the iterates around the critical set. Results of this type show a strong stability property of the deterministic subgradient descent.

In practical settings, when the function $f$ is either unknown or computation of its gradient is expensive, the deterministic gradient descent is often replaced by its stochastic version, in many cases, this may lead to a faster convergence (see e.g. \cite{bottou2018optimization}). Proximal methods, on the other hand, along with the regularizer function $g$, are widely used to regularize the initial problem of minimizing $f$. Depending on the choice of $g$, we can, for instance, preserve the boundedness of the iterates \cite{duchi2018stochastic} or promote the sparsity of solutions \cite{tibs_lasso}. It is therefore interesting  to establish stability results of the type \cite{bol_zer_pau20} for the SPGD.

In this work we investigate further the questions of oscillations of the SPGD. Our contributions are threefold. First, we show that the time spent by the SPGD to move from one accumulation point to another goes to infinity. Second, we establish an oscillation-type behavior of the drift. These two results extend \cite[Th. 7.]{bol_zer_pau20} to a stochastic and a proximal setting. Finally, our technique of proof doesn't rely on the theory of closed measures used in \cite{bol_zer_pau20} but is build upon the classical works on stochatic approximation of differential inclusions (\cite{dav-dru-kak-lee-19}, \cite{duchi2018stochastic}, \cite{ben-hof-sor-05}). We feel that this approach gives a simpler proof and allows us to treat the deterministic, the stochastic and the proximal cases in a unified manner.\\





\noindent\textbf{Paper organisation.} In Section~\ref{sec:prelim}, we recall some known facts about Clarke subgradient, path differentiable functions and differential inclusions. Our main results are given in Section~\ref{sec:main}. Section~\ref{sec:proof} is devoted to proofs.

\section{Preliminaries}\label{sec:prelim}
\subsection{Notations}
For $S \subset \bbR^d$, we denote $\cl S$ its closure and $\conv S$ its closed convex hull. For a function $F : \bbR^d \rightarrow \bbR$, we denote $\nabla F$ its gradient. Constants will usually be denoted as $C, C_1, C_2 \dots$, they can change from line to line. For a sequence $(x_n)$, we denote $\acc \{ x_n\}$ its set of accumulation points. The space of continuous functions from $\bbR_+$ to $\bbR^d$ will be denoted as $\cC(\bbR_+, \bbR^d)$, we endow this set with $\bs d$ the metric of uniform convergence on compact intervals, that is to say

\begin{equation}
  \bs d(\sx_n, \sy) \xrightarrow[n \rightarrow + \infty]{} 0 \iff \forall T >0, \sup_{h \in [0, T]} \norm{\sx_n(h) - \sy(h)} \xrightarrow[n \rightarrow + \infty]{} 0 \, .
\end{equation}

\subsection{Perturbed solutions of differential inclusions}\label{sec:pert_di}
We say that $\sH : \bbR^d \rightrightarrows \bbR^d$ is a set valued map if for each $x \in \bbR^d$ we have that $\sH(x)$ is a subset of $\bbR^d$. Consider the DI:
\begin{equation}\label{eq:di_def}
  \dot{\sx}(t) \in \sH(\sx(t))\, .
\end{equation}
We say that an absolutely continuous curve (a.c.) $\sx : \bbR_+ \rightarrow \bbR^d$ is a solution to~\eqref{eq:di_def} starting at $x \in \bbR^d$, if $\sx(0) = x$ and Eq.~\eqref{eq:di_def} holds for almost every $t\geq 0$, we denote $\sS_{\sH(x)}$ the set of these solutions and $\sS_{\sH} = \cup_{x \in \bbR^d}\sS_{\sH}(x)$.
 We say that $\sH$ is \emph{upper semi continuous at a point $x \in \bbR^d$ }
if for every $U$ a neighborhood of $\sH(x)$, there is $\delta > 0$ such that $\norm{y - x} \leq \delta \implies \sH(y) \subset U$. We say that $\sH$ is \emph{upper semi continuous} (usc) if it is upper semicontinuous at every point. We have the following existence result.
\begin{theorem}[\cite{aub-cel-(livre)84}]\label{th:aub_cel}
  Assume that, for each $x$ in $ \bbR^d$, $\sH(x)$ is nonempty, convex and compact, and there is a constant $C \geq 0$ s.t. $\sup \{ \norm{v}: v \in \sH(x)\} \leq C ( 1+ \norm{x})$. Assume that $\sH$ is usc, then for every $x  \in \bbR^d$,
   the set $\sS_{\sH(x)}$ is nonempty.
\end{theorem}

For a set-valued map $\sH$ and $\delta >0$, we denote $\sH^{\delta}(x) = \{ v \in \sH(y) :  \norm{y -x} \leq \delta\}$. In this work we will be interested in perturbed solutions of the DI associated to $\sH$.

\begin{definition}[{\cite{ben-hof-sor-05}}]
  Assume that $\sH = \sum_{i=1}^l \sH_i$, where each $\sH_i$ is a set valued map from $\bbR^d$ to $\bbR^d$. We say that an a.c. curve $\sX \in \cC(\bbR_{+} , \bbR^d)$ is a perturbed solution of the DI associated to $\sH$ if the following holds.
  \begin{enumerate}[i)]
    \item There is a function $\ssb : \bbR_+ \rightarrow \bbR_+$ and a locally integrable function $\rho : \bbR_+ \rightarrow \bbR^d$ s.t. for almost every $t\geq 0$, we have:
    \begin{equation*}
      \dot{\sX}(t) - \rho(t)\in \sum_{i=1}^{l}\sH_i^{\ssb(t)}(\sX(t)) \, .
    \end{equation*}
    \item  $\lim_{t \rightarrow + \infty } \ssb(t) = 0$ .
    \item For every $T >0$, we have: \begin{equation*}
      \lim_{t \rightarrow +\infty}\sup_{ 0 \leq h \leq T} \norm{\int_t^{t +h }\rho(u)\dif u} = 0 \, .
    \end{equation*}
  \end{enumerate}
\end{definition}
The following theorem states that, in some sense, when $t$ goes to infinity a perturbed solution shadows a solution of the corresponding DI. Its proof can be found in \cite[Proof of Th. 4.2]{ben-hof-sor-05}.
\begin{theorem}\label{th:ben_pert}
  Let $\sX$ be a perturbed solution associated with $\sH = \sum_{i=1}^l \sH_i$ and assume that each of the $\sH_i$ satisfies the assumptions of Th.~\ref{th:aub_cel}. Then the family $\{ \sX(t+ \cdot) : t \in \bbR_{+}\}$ is relatively compact (in $\cC(\bbR_{+} , \bbR^d)$) and we have:
  \begin{equation*}
  \lim_{t \rightarrow + \infty}  \bs d (\sX(t + \cdot), \sS_{\sH}) =0
  \end{equation*}
\end{theorem}
\begin{remark}
  Strictly speaking, in \cite{ben-hof-sor-05}, a perturbed solution to the DI $\dot{\sx}(t) \in \sH(\sx(t))$ was required to satisfy $\dot{\sX}(t) - \rho(t) \in \sH^{\ssb(t)}(\sX(t))$, where $\sH =\sum_{i=1}^{l}\sH_i$. Nevertheless, the proof of \cite[Th. 4.2]{ben-hof-sor-05} goes through with our definition.
\end{remark}

Given a set $\cA$, we say that a continuous function $\varphi : \bbR^d \rightarrow \bbR$ is a \emph{strict Lyapunov function} for the DI~\eqref{eq:di_def}, if for every $t>0$ and $\sx \in \sS_{\sH(x)}$, we have that $\varphi(\sx(t)) < \varphi(x)$ if $x \notin \cA$ and $\varphi(\sx(t)) \leq \varphi(x)$ otherwise. If such a $\varphi$ exists, then more can be said about the behavior of a perturbed solution.
\begin{theorem}[{\cite[Th. 3.6 and Prop. 3.27]{ben-hof-sor-05}}]\label{th:lyap_ben}
  Assume that we are in the context of Th.~\ref{th:ben_pert} and let $\varphi$ be a strict Lyapunov function for a set $\cA \subset \bbR^d$. Assume, moreover, that $\varphi(\cA)$ is of empty interior, we have:
\begin{equation*}
  \sL_{\sX} := \bigcap_{t \geq 0} \cl {\{\sX(u) : u \geq t\}} \subset A
\end{equation*}
and $\varphi$ is constant on $\sL_{\sX}$.
\end{theorem}

In this note, we will be primarily interested in two particular set valued maps.
Consider $f: \bbR^d \rightarrow \bbR$ a locally Lipschitz function. By Rademacher's theorem, $f$ is differentiable almost everywhere. The set
$\partial f(x)$ of Clarke subgradients of $f$ at $x$ is defined as follows:
\begin{equation}
  \partial f(x) = \conv \{ v : \textrm{ there is a sequence } x_i \rightarrow x \textrm{ s.t. $f$ is differentiable at $x_i$ and } \nabla f(x_i) \rightarrow v\}\, .
\end{equation}
The set $\{ x: 0 \in \partial f(x)\}$ of Clarke-critical points contains local extrema (see \cite{cla-led-ste-wol-livre98}). The map $\partial f : \bbR^d \rightrightarrows \bbR^d$ is usc and for every $x$ in $ \bbR^d$, $\partial f(x)$ is nonempty, compact and convex.\\
Given a convex set $\cX \subset \bbR^d$, the normal cone of $\cX$ is a set valued map $\cN_{\cX} : \bbR^d \rightrightarrows \bbR^d$, defined as:
\begin{equation}
  \cN_{\cX}(x) = \{ v: \scalarp{v}{y -x} \leq 0, \forall y \in \cX\} \, .
\end{equation}
For each $x \in \cX$, $\cN_{\cX}(x)$ is a closed convex subset of $\bbR^d$.

\subsection{Semialgebraic and definable functions}
An important case to which our results apply, is when $f, g$ and $\cX$ are semialgebraic, or more generally definable. We say that a set $A \subset \bbR^{N}$ is semialgebraic if it can be written as a finite union and intersection of sets of the form $\{x: P(x) \leq 0 \}$, where $P : \bbR^{N} \rightarrow \bbR$ is some polynomial. A function is semialgebraic if its graph is a semialgebraic set. While they may be nonsmooth, semialgebraic functions present strong regularity properties. Among other things, they are $C^k$ differentiable on a dense open set (for any $k \geq 0$) and stable under many elementary operations such as composition, sum, multiplication.

A generalization of this notion, which preserve the aformentioned structural properties, is the one of definableness in an o-minimal structure. While we will not mathematically define this notion here, let us mention that any semialgebraic function, as well as the exponential and the logarithm are definable (hence, also their composition). This explains their ubiquity in the optimization literature. Up to our knowledge, the first work to exploit the link between optimization and definableness was \cite{bolte2007clarke}. An interested reader can find more on definability
and its usefulness in optimization in \cite{iof08}, and more details
in \cite{van96}, \cite{coste2000introduction}, \cite{dav-dru-kak-lee-19}.


\subsection{Path differentiable functions}\label{sec:path_diff}
We say that a locally Lipschitz function $f : \bbR^d \rightarrow \bbR$ is \emph{path differentiable} if for any a.c. curve $\sx : [0,1] \rightarrow \bbR^d$, for almost every $t \in [0, 1]$:
\begin{equation}\label{eq:path_dif}
  (f \circ \sx)'(t) = \scalarp{v}{\dot{\sx}(t)} \quad \forall v \in \partial f(\sx(t)).
\end{equation}
By \cite[Proposition 2]{bolte2019conservative}, every convex, concave, semialgebraic or definable function is path differentiable. Moreover, if another function $g \colon \bbR^d \rightarrow \bbR$ is path differentiable, then $f + g$ is also path differentiable \cite[Corollary 4]{bolte2019conservative}. From a similar point of view, if $\cX$ is a convex set, then for any a.c. curve $\sx : [0,1] \rightarrow \bbR^d$, for almost every $t \in [0, 1]$:
\begin{equation}\label{eq:con_path_dif}
  \scalarp{v}{\dot{\sx}(t)} = 0 \quad  \forall v \in \cN_{\cX}(\sx(t))\, .
\end{equation}
Consider now $f,g : \bbR^d \rightarrow \bbR$ path differentiable, $\cX \subset \bbR^d$ a convex set and $\sx$ a solution to the DI~\eqref{eq:DI_prox}. Using Eq.~\eqref{eq:path_dif} and \eqref{eq:con_path_dif} and the fact that $\partial (f+g) \subset \partial f + \partial g$, we obtain
\begin{equation}\label{eq:comp_grad_lyap}
  (f + g)(\sx(t)) - (f+g)(\sx(0)) = -\int_{0}^t \norm{\dot{\sx}(u)}^2 \dif u \, .
\end{equation}
This implies that $(f+g)(\sx(t)) < (f+g)(\sx(0))$ if $\sx(0) \notin \cZ$. In other words, $f+g$ is a strict Lyapunov function for the DI~\eqref{eq:DI_prox}.

\section{Main results}\label{sec:main}
Consider $(\Omega, \Xi, \bbP)$ a probability space and $(\eta_n)$ a sequence of random variables with values in $\bbR^d$. Define $\prox^{\gamma}_{g, \cX} : \bbR^d \rightrightarrows \bbR^d$, the proximal operator for $g$ on $\cX$ with a step $\gamma$:
\begin{equation}\label{eq:prox}
  \prox^{\gamma}_{g, \cX}(x) = \argmin_{y \in \cX} \{ g(y) + \frac{1}{2\gamma} \norm{y-x}^2\}\, .
\end{equation}
We study Eq.~\eqref{eq:spgd_incl} under the following assumptions.
\begin{assumption}\label{hyp:model1}
  \-
    \begin{enumerate}[i)]
      \item\label{hyp:X_conv} The set $\cX$ is a closed convex subset of $\bbR^d$.
      \item\label{hyp:lip} The functions $f,g : \bbR^d \rightarrow \bbR$ are locally Lipschitz continuous.
      \item\label{hyp:marting}   There is a filtration $(\mcF_{n})_{n \in \bbN}$, such that $(\eta_n)$ is a martingale difference sequence adapted to it, and $x_n$ is $\mcF_n$ measurable for every $n \in \bbN$.
      \item\label{hyp:steps}  The sequence of stepsizes $(\gamma_n)$ is nonnegative and such that $\sum_{i=0}^{+ \infty} \gamma_i = + \infty$.
    \end{enumerate}
\end{assumption}
Note that if $g$ is nonconvex, $\prox^{\gamma}_{g, \cX}(x)$ is a set in $\bbR^d$. However, as soon as $x_{n+1}$ is chosen in a measurable manner (relatively to $\eta_{n+1}$ and $x_n$), $(x_n)$ will be adapted to $(\mcF_n)$. Such a choice is always possible (see e.g. \cite{dav-dru-kak-lee-19}).

By \cite[10.2 and 10.10]{roc-wets-livre98}, we can rewrite Eq.~\eqref{eq:spgd_incl} as:
\begin{equation}\label{eq:spgd_iter}
  x_{n+1} = x_n - \gamma_n (v_n + v_n^g + v_n^{\cX}) + \gamma_n \eta_{n+1} \, ,
\end{equation}
where $v_n^g \in \partial g(x_{n+1})$ and $v_n^{\cX} \in \cN_{\cX}(x_{n+1})$.

\begin{assumption}\label{hyp:model2}\-
  \begin{enumerate}[i)]
    \item\label{hyp:bound} Almost surely, $\sup_{n} \norm{x_n}  < + \infty$.
    \item\label{hyp:noise} There is $q \geq 2$ such that

    \begin{equation}
      \sum_{i=0}^{+\infty}\gamma_i^{1 + q/2} < + \infty \, ,
    \end{equation}
    and, for any compact set $\cK \subset \bbR^d$,
    \begin{equation}
      \sup_{n \in \bbN} \bbE[\norm{\eta_{n+1}}^q \1_{x_n \in \cK}| \mcF_n] < + \infty \, .
    \end{equation}
  \end{enumerate}
\end{assumption}
Assumptions of this type are standard in the field of stochastic approximation. Assumption~\ref{hyp:model2}-\eqref{hyp:bound} prevent the algorithm to diverge. Note that it is superfluous if $\cX$ is compact. Otherwise it can be obtained by a proper choice of the regularizer $g$ (see \cite{duchi2018stochastic}).

Let $\tau_{n} = \sum_{i=1}^n \gamma_i$ be the discrete time of the algorithm. Define the linearly interpolated process $\sX \in \cC( \bbR_{+}, \bbR^d)$ by:
\begin{equation*}
  \sX(t)  = x_n + \frac{t - \tau_n}{\gamma_{n+1}} (x_{n+1}-x_n) \quad \textrm{for}  \quad \tau_n \leq t < \tau_{n+1} \, .
\end{equation*}
Following \cite{ben-hof-sor-05} we will show that $\sX$ is a perturbed solution of the DI~\eqref{eq:DI_prox}.
The next two assumptions ensure us that $f+g$ will be a Lyapunov function for the DI~\eqref{eq:DI_prox}.
\begin{assumption}\label{hyp:path_dif}
  The functions $f$ and $g$ are path differentiable.
\end{assumption}
\begin{assumption}\label{hyp:Sard}
  The set of Clarke critical values $ \{f(x) + g(x) : x \in \cZ\}$ has an empty interior.
\end{assumption}

Assumption~\ref{hyp:Sard} is a classical Sard-type condition. It ensures the fact that if $\sx$ is a solution to the DI~\eqref{eq:DI_prox}, with $\sx(0) \in \cZ$, then $\sx$ is constant. As established in~\cite{bolte2007clarke}, it is satisfied as soon as $f,g$ and $\cX$ are definable.

The next two propositions are not new and can be found in one way or another in e.g. \cite{dav-dru-kak-lee-19}, \cite{bolte2019conservative}, \cite{maj-mia-mou18}, \cite{bol_zer_pau20}. Nevertheless, since our set of assumptions is slightly different and their proof is a simple application of Section~\ref{sec:pert_di}, for completeness, we include it in Section~\ref{sec:proof_apt}.
\begin{proposition}
  \label{pr:apt}
  Let Assumptions~\ref{hyp:model1} and \ref{hyp:model2} hold, then the family $(\sX(t + \cdot))_{t \geq 0}$ is relatively compact. Moreover, if a sequence $t_n \rightarrow + \infty$ and $\sx \in \cC(\bbR_{+}, \bbR^d)$ is such that $\boldsymbol d(\sX(t_n + \cdot), \sx) \rightarrow 0$, then $\sx$ is a solution to the DI~\eqref{eq:DI_prox}.
\end{proposition}

\begin{proposition}
  \label{pr:x_conv}
  Under Assumptions~\ref{hyp:model1}--\ref{hyp:Sard}, the set $\acc \{x_n\}$ is included in $\cZ$ and $f + g$ is constant on $\acc \{ x_n\}$.
\end{proposition}

The next theorem  tells us that even if $\acc \{ x_n \}$ is not a single point, the time that it takes to $(x_n)$ to go from one accumulation point to another goes to infinity. This is an extension of \cite[Th. 6.i), Th. 7.i)]{bol_zer_pau20}, to the best of our knowledge this result is new in a stochastic and proximal setting.
\begin{theorem}\label{th:two_sol_T}
  Let Assumptions~\ref{hyp:model1}--\ref{hyp:Sard} hold. Let $x,y$ be two distinct points in $\acc \{ x_n\}$. Consider two sequences $n_i, n_j $, with $n_i \leq n_j $, such that $x_{n_i} \rightarrow x$ and $x_{n_j}\rightarrow y$. Then $\tau_{n_j} - \tau_{n_i} \rightarrow + \infty$.

  Under Assumptions~\ref{hyp:model1}--\ref{hyp:path_dif}, the same result is true if $(f+g)(x) \leq (f+g)(y)$.
\end{theorem}
As it is shown in \cite{zer20}, it is possible that $\acc \{ x_n\}$ is not reduced to a unique point. Nevertheless, Th.~\ref{th:two_sol_T} implies that the ``nonconvergence" happens in a very slow manner. Asymptotically, the time spent by the algorithm to move from one accumulation point to another goes to infinity.

We now investigate the question of oscillations. Given $U, V$ two open sets, such that $\cl{U} \subset V$, we will call $I = [n_1, n_2]$ a maximal interval related to $U,V$ if the set $X_{n_1}^{n_2} := \{ x_{n_1}, x_{n_1 + 1}, \dots, x_{n_2}\}$ is such that $X_{n_1}^{n_2} \subset V$, $X_{n_1}^{n_2} \cap U \neq \emptyset$ and either $x_{n_1 -1}$ or $x_{n_2 + 1}$ is not in $V$.
 The next two results are an extension of \cite[Th. 7]{bol_zer_pau20} to a stochastic setting.

\begin{theorem}[Long intervals]\label{th:long_i}
  Let Assumptions~\ref{hyp:model1}-\ref{hyp:Sard} hold. Consider $x \in \acc\{ x_n\}$ and $U,V$ two neighborhoods of $x$ such that $\cl{U} \subset V$. For $i \in\bbN$, denote $I_i = [n_{i_1}, n_{i_2}]$ a sequence of distinct maximal intervals related to $U, V$. Then, either one of $I_i$ is unbounded or $\tau_{n_{i_2}} - \tau_{n_{i_1}} \rightarrow + \infty$.
\end{theorem}

\begin{theorem}[Oscillation compensation]\label{th:main}
  Let Assumptions~\ref{hyp:model1}-\ref{hyp:Sard} hold, and fix
 $U$, $V$  and $I_i$ as in Th.~\ref{th:long_i}. Denote $A = \bigcup I_i$, then
    \begin{equation}\label{eq:osc_app}
      \frac{\sum_{i=1}^n \gamma_{i} (v_i + v_i^g + v_i^{\cX}) \1_{A}(x_i)}{ \sum_{i=1}^n \gamma_{i} \1_{A}(x_i)} \xrightarrow[n \rightarrow + \infty]{} 0 \, .    \end{equation}
\end{theorem}
Th.~\ref{th:main} gives an intuitive explanation of why Th.~\ref{th:two_sol_T} holds. Indeed, while the drift coming from one iteration $v_i + v_i^g + v_i^{\cX}$ might not go to zero (as it happens for such a simple example as $f(x) = \norm{x}$, $g =0$ and $\cX = \bbR^d$), on average, it compensates itself. Th.~\ref{th:two_sol_T} and \ref{th:main} suggest that the algorithm oscillates around its accumulation set, while the center of these oscillations moves in $\acc\{x_n \}$ with a vanishing speed.

Let us finish with a remark on the Eq.~\eqref{eq:osc_app}. At first sight, maximal intervals in Th.~\ref{th:main} and Th.~\ref{th:long_i} may seem artificial. A more satisfactory result would be
\begin{equation}\label{eq:osc_true}
  \frac{\sum_{i=1}^n \gamma_{i} (v_i + v_i^g + v_i^{\cX}) \1_{U}(x_i)}{ \sum_{i=1}^n \gamma_{i} \1_{U}(x_i)} \xrightarrow[n \rightarrow + \infty]{} 0 \, ,
\end{equation}
where $U$ is an open neighborhood of an accumulation point $x$. Looking at the proof of Th.~\ref{th:main}, to obtain Eq.~\eqref{eq:osc_true}, we could think of defining maximal intervals as $I_i = [n_{i_1}, n_{i_2}]$ such that $\{ x_{n_{i_1}}, \dots, x_{n_{i_2}}\} \subset U$ and $x_{n_{i_1} -1}, x_{n_{i_2}+ 1} \notin U$. Unfortunately, for this type of intervals we dont have an equivalent of Th.~\ref{th:long_i}, i.e. it may very well be that the quantity $\tau_{n_{i_2}} - \tau_{n_{i_1}}$ is bounded. Actually, it is not very hard to show, that for the function from \cite[Section 2]{zer20}, there are $x, U$ such that Eq.~\eqref{eq:osc_true} is false.

Nevertheless,  as explained in \cite{bol_zer_pau20}, Eq.~\eqref{eq:osc_app} is a good approximation of Eq.~\eqref{eq:osc_true}. Indeed, apply Th.~\ref{th:main} with $U$ and $V=U^{\delta}$, where $U^{\delta} = \{ y \in \bbR^d : \exists z \in U, \norm{z - y} < \delta\}$, then, as an approximation, we have
\begin{equation*}
 \lim_{\delta \rightarrow 0} \lim_{n \rightarrow + \infty} \frac{\sum_{i=1}^n \gamma_{i+1}(v_i + v_i^g + v_i^{\cX}) \1_{A}(x_i)}{ \sum_{i=1}^n \gamma_{i+1} \1_{A}(x_i)} \approx   \lim_{n \rightarrow + \infty}\frac{\sum_{i=1}^n \gamma_{i+1} (v_i + v_i^g + v_i^{\cX}) \1_U(x_i) }{\sum_{i=1}^n \gamma_{i+1} \1_{U}(x_i)} \, .
\end{equation*}

\section{Proofs}\label{sec:proof}

In the following we will denote $x_{n+1/2} = x_n - \gamma_n v_n + \gamma_{n} \eta_{n+1}$ and
\begin{equation}\label{eq:defNTn}
  N(T, n) = \inf \{ j \geq n \textrm{ s.t. } \tau_j - \tau_n \geq T\} \, .
\end{equation}
\subsection{Proof of Prop.~\ref{pr:apt} and \ref{pr:x_conv}}\label{sec:proof_apt}

To put ourselves in the context of Section~\ref{sec:pert_di} we need to alter the map $- \partial f- \partial g - \cN_{\cX}$ in a way that it verifies assumptions of Th.~\ref{th:ben_pert} and \ref{th:lyap_ben}. While this section is slightly technical, conceptually, we just find a set-valued map $G$ verifying assumptions of Th.~\ref{th:lyap_ben} and s.t. $x_{n+1} \in G(x_n)$. A convinced reader may want to skip to Section~\ref{subsec:two_sol_t_proof}.\\
We start with two technical lemmas.

\begin{lemma}\label{lm:noise_conv}
  Under Assumptions~\ref{hyp:model1} and \ref{hyp:model2}, almost surely, for every $T >0$, we have:
  \begin{equation}\label{eq:noise_conv}
    \lim_{n \rightarrow + \infty}  \sup_{n \leq j \leq N(T,n)} \norm{ \sum_{i=n}^j \gamma_i \eta_{i+1}} = 0 \, .
  \end{equation}
  As a consequence, the sequence $(\norm{x_{n+1/2}})$ is almost surely bounded.
\end{lemma}
\begin{proof}
  Indeed, since almost surely $\sup \norm{x_n} < + \infty$, for each $\delta > 0$, there is $C > 0$ s.t. if we denote $A = \{\forall n \in \bbN \norm{x_n} \leq C\}$, then $\bbP(A) > 1 - \delta$. Define $\tilde{\eta}_{n+1} = \eta_{n+1} \1_{\norm{x_n} \leq C}$, then $\bbE[\tilde{\eta}_{n+1} | \mcF_n] = 0$ and $\sup_{n \in \bbN}\bbE[\norm{\tilde{\eta}_{n+1}}^q] < + \infty$.
  Hence, by \cite[Prop. 4.2]{ben-(cours)99}, we have $\sup_{n \leq j \leq N(T,n)} \norm{ \sum_{i=n}^j \gamma_i \tilde{\eta}_{i+1}} \xrightarrow[n \rightarrow + \infty]{} 0$. Since $\delta$ is arbitrary, Eq.~\eqref{eq:noise_conv} follows.
\end{proof}

\begin{lemma}\label{lm:alt_di}
  Let Assumptions~\ref{hyp:model1} and~\ref{hyp:model2} hold. Let $A \in \Xi$ be a probability one set on which $(x_n)$ and $(x_{n+1/2})$ are bounded, and let $C$ be a random variable s.t.
  $\norm{x_n} < C$ and $C$ is finite valued on $A$. Then, for each  $\omega \in A$, there are two globally Lipschitz functions $\tilde{g}, \tilde{f} : \bbR^d \rightarrow \bbR$ and a bounded set-valued map $\widetilde{\cN}_{\cX} : \bbR^d \rightrightarrows \bbR^d$ s.t. in Eq.~\eqref{eq:spgd_iter} we have
   $v_n(w) \in \partial \tilde{f}(x_n(w))$, $v_n^g(w) \in \partial \tilde{g}(x_{n+1}(w))$ and $v_n^{\cX}(w) \in \widetilde{\cN}_{\cX}(x_{n+1}(w))$.\\
   Moreover, if $\sx$ is a solution to the DI:
   \begin{equation}\label{eq:di_alter}
     \dot{\sx}(t) \in - \partial \tilde{f}(\sx(t))- \partial \tilde{g}(\sx(t)) - \widetilde{\cN}_{\cX}(\sx(t)) \, ,
   \end{equation}
   and that $\sx$ remains in $B(0, C) \cap \cX$, then $\sx$ is a solution to the DI~\eqref{eq:DI_prox}.\\
   Finally, denoting $\widetilde{\cZ} = \{ x : 0 \in \partial \tilde{f}(x) + \partial \tilde{g}(x) + \widetilde{\cN}_{\cX}(x)\}$, we have the equality $\widetilde{\cZ} \cap B(0, C) = \cZ \cap B(0, C)$.
\end{lemma}
\begin{proof}
   Let $\Pi_{C+1} : \bbR^d \rightarrow \bbR^d$ be the projection on $B(0, C+1)$. Define $\tilde{f}(x) = f(\Pi_{C+1}(x))$, $\tilde{g}(x) = g(\Pi_{C + 1}(x))$. By construction, we have that $v_n \in \partial \tilde{f}(x_n)$ and $v_n^g \in \partial g(x_{n+1})$ and  that $v_n, v_n^g$ are bounded by $L_{\tilde{f}}, L_{\tilde{g}}$
   the Lipschitz constants of $\tilde{f}$ and $\tilde{g}$.
  Hence, since $x_{n+1/2}$ is bounded, there is $C_2$ s.t.
   $ \sup \{ \norm{v_n^{\cX}} : n \in \bbN\} < C_2$.
  Defining $\widetilde{\cN}_{\cX}(x) = \{ v: \norm{v} \leq \max(C_2,L_f, L_g), v \in \Pi_{\cX}(x)\}$, where $\Pi_{\cX}$ is a projection on $\cX$, proves the first claim. The two other statements immediately follow from our construction.
\end{proof}
 To prove Prop.~\ref{pr:apt} it remains to show that $\sX$ is a perturbed solution to the DI~\eqref{eq:di_alter}.
 For $t \in [\tau_n, \tau_{n+1})$, we define $\rho(t) = \eta_{n+1}$ and $\ssb(t) = \norm{x_{n+1} - x_n}$. The condition on $\rho$ immediately follows from Lemma~\ref{lm:noise_conv}. The condition on $\ssb$ follows from the following lemma.

\begin{lemma}
  Under Assumptions~\ref{hyp:model1} and \ref{hyp:model2}, almost surely, we have that $\norm{x_{n+1} - x_n} \xrightarrow[n \rightarrow + \infty]{} 0$.
\end{lemma}
\begin{proof}
  By Lemma~\ref{lm:noise_conv}, we have that $\norm{x_{n+1/2} - x_n} \xrightarrow[n \rightarrow + \infty]{} 0$, moreover, we have:
\begin{equation*}
  g(x_{n+1}) + \frac{1}{2 \gamma_n} \norm{x_{n+1} - x_{n+1/2}}^2 \leq g(x_n) + \frac{1}{2 \gamma_n}\norm{x_{n} - x_{n+1/2}}^2\, .
\end{equation*}
Therefore,
\begin{equation*}
  \begin{split}
      \frac{1}{2 \gamma_n} \norm{x_{n+1} -x_{n}}^2 &\leq g(x_{n}) - g(x_{n+1}) - \frac{1}{\gamma_n}\scalarp{x_{n+1} - x_n}{x_n - x_{n+1/2}}\\
       &\leq  \norm{x_{n+1} - x_n}\left(L_g + \frac{\norm{x_n - x_{n+1/2}}}{\gamma_n}\right) \, ,
  \end{split}
\end{equation*}
and
\begin{equation*}
  \norm{x_{n+1}- x_n} \leq \gamma_n L_g + \norm{x_n - x_{n+1/2}} \, ,
\end{equation*}
which finishes the proof.
\end{proof}

To finish the proof of Prop.~\ref{pr:apt} consider $t_n \rightarrow + \infty$ and $\sx$ s.t. $ \boldsymbol d(\sX(t_n+ \cdot), \sx) \rightarrow 0$. Then, by Th.~\ref{th:ben_pert}, $\sx$ is a solution to the DI~\eqref{eq:di_alter}, moreover, it remains in $B(0, C) \cap \cX$, therefore, it is also a solution to the DI~\eqref{eq:DI_prox}.

For the proof of Prop.~\ref{pr:x_conv}, notice that $\tilde{f} + \tilde{g}$ is path differentiable (as a composition of path differentiable functions). Then, in the same way as in Section~\ref{sec:path_diff}, we have that $\tilde{f} + \tilde{g}$ is a strict Lyapunov function for the DI~\eqref{eq:di_alter} and for the set $\tilde{\cZ}$. Since $\acc \{ x_n \} = \sL_{\sX} \subset \cl{B(0, C)}$,
by Th.~\ref{th:lyap_ben} we have that $\sL_{\sX} \subset \widetilde{\cZ}\cap \cl{B(0,C)} \subset \cZ$, and that $f + g$ is constant on $\acc \{ x_n\}$.

\subsection{Proof of Th.~\ref{th:two_sol_T}}\label{subsec:two_sol_t_proof}

\begin{lemma}\label{lm:sol_const}
  Let Assumptions~\ref{hyp:model1}--~\ref{hyp:path_dif} hold, let $\tau_n$ be a positive sequence, with $\tau_n \rightarrow + \infty$, and $\sx$ s.t. $\sX(\tau_n + \cdot) \rightarrow \sx$, then
  \begin{equation}
    (f+g)(\sx(h)) \leq (f+g)(\sx(0)), \quad \forall h \in \bbR_+ \, .
  \end{equation}
Moreover, if for some $h \geq 0$, $(f+g)(\sx(h)) = (f+g)(\sx(0))$, then $\sx(h') = \sx(0)$ for every $h' \in [0, h]$.
  If additionally Assumption~\ref{hyp:Sard} holds, then:
  \begin{equation}\label{eq:sol_const}
    \sx(h) = \sx(0) , \quad \forall h \in \bbR_+ \, .
  \end{equation}
\end{lemma}
\begin{proof}
  By Prop.~\ref{pr:apt}, $\sx$ is a solution to the DI~\eqref{eq:DI_prox}, and the first result follows by Eq.~\eqref{eq:comp_grad_lyap}.\\
   Under Assumption~\ref{hyp:Sard}, we have that $\sx(\bbR_+) \subset \acc \{x_n\} \subset \cZ$,
   hence, by Prop.~\ref{pr:x_conv}, we have that $(f+g)\circ\sx$ is constant. Using Assumption~\ref{hyp:path_dif}, we have for all $h \in \bbR_+$,
  \begin{equation}\label{eq:F_const}
    0 = (f+g)(\sx(h)) - (f+g)(\sx(0)) = - \int_{0}^h \norm{\dot{\sx}(u)}^2 \dif u \, .
  \end{equation}
  This implies that $ \int_{0}^h \norm{\dot{\sx}(u)}^2 \dif u = 0$. Hence, $\dot{\sx}(h) = 0$ for almost every $h \in \bbR_{+}$ and we obtain Eq.~\eqref{eq:sol_const}.

\end{proof}

  Suppose that there is $T > 0$ such that $\tau_{n_j} - \tau_{n_i} \leq T$. The sequence $\sX(\tau_{n_i} + \cdot)$ is relatively compact, and after extraction it converges to $\sx$ a solution to~\eqref{eq:DI_prox}. Extract once again to have $\tau_{n_j} - \tau_{n_i} \rightarrow h$. Then
  \begin{equation*}
    \sX(\tau_{n_j}) - \sX(\tau_{n_i}) \rightarrow \sx(h) - \sx(0) = y - x \, ,
  \end{equation*}
and we obtain a contradiction with Lemma~\ref{lm:sol_const}.

\subsection{Proof of Th.~\ref{th:long_i}}
The next lemma is the key ingredient for the proofs of Th.~\ref{th:long_i} and Th.~\ref{th:main}.
\begin{lemma}\label{lm:int_zer}
  Under Assumptions~\ref{hyp:model1}--\ref{hyp:Sard}, we have
  \begin{equation*}
    \sup_{ n \leq j \leq N(T,n)} \norm{\sum_{i=n}^j \gamma_i (v_i + v_i^g + v_i^{\cX})} \xrightarrow[n \rightarrow + \infty] {} 0\, .
  \end{equation*}
\end{lemma}
\begin{proof}

    Suppose that we have  $\varepsilon > 0$ and two sequences $n_k$ and $n_k \leq j_k \leq N(T, n_k)$, such that for $n_k$ large enough:
    \begin{equation*}
      \norm{\sum_{i=n_k}^{j_k}  \gamma_i (v_i + v_i^g + v_i^{\cX})} > \varepsilon \, .
    \end{equation*}
    This implies:
    \begin{equation*}
      \norm{x_{j_k} - x_{n_k} + \sum_{i=n_k}^{j_k} \gamma_i \eta_{i+1}} > \varepsilon \, .
    \end{equation*}
    Extract a sequence such that $\sX(\tau_{n_k}  + \cdot)$ converges to $\sx$ and $\tau_{j_k} - \tau_{n_k} \rightarrow h$, with $h \leq T$. Then $x_{j_k} \rightarrow \sx(h)$ and $x_{n_k} \rightarrow \sx(0)$, but
    $\norm{\sx(h) - \sx(0)} \geq \varepsilon$ which is impossible by Lemma~\ref{lm:sol_const}.
\end{proof}

  Suppose that no $I_i$ is unbounded, then we can choose $n_i \in I_i = [n_{i_1}, n_{i_2}]$ such that $x_{n_i} \in U$. Since $x_{n_{i_2 + 1}}$ is in $V^c$, after extraction $x_{n_i} \rightarrow y_1$ and $x_{n_{i_2 + 1}} \rightarrow y_2$, with $y_2 \neq y_1$, moreover:
  \begin{equation}
     \tau_{n_{i_2+1}}  - \tau_{n_{i}}  - \gamma_{n_{i_2+1}}  \leq \tau_{n_{i_2}} - \tau_{n_{i_1}} \, .
  \end{equation}
  By Th.~\ref{th:two_sol_T}, the first term of this inequality tends to infinity.

  \subsection{Proof of Th.~\ref{th:main}}
    Take $I_i$ as in Th.~\ref{th:long_i}, and $A_n = \bigcup_{i \leq n} I_i$.
    Define
    \begin{equation*}
      u_n = \frac{a_n}{b_n} =   \frac{\sum \gamma_i (v_i + v_i^g + v_i^{\cX})\1_{A_n}(x_i)} {\sum \gamma_i \1_{A_n}(x_i)}\, .
    \end{equation*}
    Then,
    \begin{equation}\label{eq:un}
      u_{n+1} = \frac{a_n + \sum \gamma_i( v_i + v_i^g + v_i^{\cX})\1_{I_{n+1}}(x_i)}{b_n +  \sum \gamma_i \1_{I_{n+1}}(x_i)} \, .
    \end{equation}
 Fix $\varepsilon > 0$, by Lemma~\ref{lm:int_zer}, there is $n_0$ such that, for $n \geq n_0$, $\norm{\sum_{i=n_k}^{j_k} \gamma_i (v_i + v_i^g + v_i^{\cX})} \leq \varepsilon $.  Decompose $I_i = [n_{i_1}, n_{i_2}] = \bigcup_{1 \leq k \leq K_i} [a_{i, k}, a_{i, k +1}]$, with $a_{i,1} = n_{i_1}$ and $a_{i,k +1} = \min\{ N(T, a_{i,k}), n_{i_2} \}$.
We obtain:
    \begin{equation*}
      \begin{split}
        u_{n+1} &= \frac{a_n + \sum_{k\leq K_n}\sum_{i =a_{n,k}}^{a_{n,k+1}} \gamma_i (v_i + v_i^g + v_i^{\cX})}{b_n + \sum_{k \leq K_n}\sum_{i = a_{n,k}}^{a_{n,k +1}} \gamma_i}\\
         &\leq \frac{a_n + (K_n)\varepsilon}{b_n + (K_n-1) T}\, .
      \end{split}
    \end{equation*}
      By Th.~\ref{th:long_i}, we have that $K_n \rightarrow + \infty$ and, therefore, for $n$ large enough:
      \begin{equation*}
        u_{n+1} \leq \frac{a_n + 2(K_n - 1)\varepsilon}{b_n + (K_n-1) T} \, .
      \end{equation*}
      Hence, by induction:
    \begin{equation*}
      u_{n+j} \leq \frac{a_n + 2\varepsilon \sum_{k = n}^{n+j -1}(K_i- 1) }{b_n + T\sum_{k = n}^{n+j -1} (K_i -1)} \, .
    \end{equation*}
 Therefore, $\lim u_n \leq \frac{2\varepsilon}{T}$. Since $\varepsilon$ is arbitrary, this finishes the proof.

\section*{Fundings}
This work was supported by R\'egion Île-de-France.

\bibliographystyle{plain}
\bibliography{math}
\end{document}